\documentclass[12pt,twoside,final,psamsfonts]{amsart}
\usepackage[psamsfonts]{amssymb}
\usepackage{times,a4wide}
\usepackage[draft=false]{hyperref}
\usepackage{url}

\theoremstyle{plain}
\newtheorem{thm}{Theorem}

\newtheorem{cor}[thm]{Corollary}
\newtheorem{lemma}[thm]{Lemma}
\theoremstyle{remark}
\newtheorem*{rem}{Remark}
\newtheorem*{rems}{Remarks}
\theoremstyle{definition}
\newtheorem{defn}{Definition}

\newcommand{\C}{{\mathbb C}}
\newcommand{\R}{{\mathbb R}}
\newcommand{\calf}{\mathcal{F}}
\newcommand{\iC}{{\int_{\C}}}

\newcommand{\Fpa}{{\mathcal F_{\phi}^p}}
\newcommand{\Fdosa}{{\mathcal F_{\phi}^2}}
\newcommand{\Finfa}{{\mathcal F_{\phi}^\infty}}

\newenvironment{rlist}
{

\begin{enumerate}}
{\end{enumerate}}

\title{Traces of functions in Fock spaces on lattices of critical density.}
\thanks{The authors are supported by the project MTM2008-05561-C02-01 and the grant
2009 SGR 1303}
\author{Jeremiah Buckley}
\address{Dept.\ Matem\`atica Aplicada i An\`alisi,
Universitat  de Barcelona, Gran Via 585, 08007 Bar\-ce\-lo\-na, Spain} \email{jerry.buckley07@gmail.com}
\author{Xavier Massaneda}
\address{Dept.\ Matem\`atica Aplicada i An\`alisi,
Universitat  de Barcelona, Gran Via 585, 08007 Bar\-ce\-lo\-na, Spain} \email{xavier.massaneda@ub.edu}
\author{Joaquim Ortega-Cerd\`a}
\address{Dept.\ Matem\`atica Aplicada i An\`alisi,
Universitat  de Barcelona, Gran Via 585, 08007 Bar\-ce\-lo\-na, Spain} \email{jortega@ub.edu}

\begin{document}
\begin{abstract}
Following a scheme of Levin we describe the values that functions in Fock spaces take on lattices of critical density in terms of both the size of the
values and a cancelation condition that involves discrete versions of the Cauchy and Beurling-Ahlfors transforms.
\end{abstract}
\maketitle

\section{Introduction}
We are interested in describing the set of values $c=(c_\lambda)_{\lambda\in\Lambda}$ such that there exists some function $f$ in a Fock space satisfying
the condition $f|\Lambda=c$ where $\Lambda$ is a sequence that `just fails' to be interpolating. While we shall prove our results in a more general
context, we begin by introducing the problem in the classical Bargmann-Fock space, where the results are more easily digestible.

We define the classical Bargmann-Fock space, as studied in \cite{Se92} and \cite{SeWa}, by
$$\calf^p =\{f\in H(\mathbb{C}):\|f\|_{\calf^p}^p=\iC|f(z)|^p e^{-p|z|^2}dm(z)<+\infty\}\text{, for }1\leq p<+\infty$$
and
$$\calf^\infty =\{f\in H(\mathbb{C}):\|f\|_{\calf^\infty}=\sup_{z\in\mathbb{C}}|f(z)|e^{-|z|^2}<+\infty\}$$
where $m$ denotes the Lebesgue measure on the plane. Seip and Wallst{\'e}n completely characterised sets of sampling and sets of interpolation in these
spaces. We begin with a definition.
\begin{defn}
A sequence $\Lambda\subseteq\mathbb{C}$ is an \textit{interpolating sequence for $\calf^p$, where $1\leq p<+\infty$} if for every sequence of values
$c=\left(c_\lambda\right)_{\lambda\in\Lambda}$ such that
$$\sum_{\lambda\in\Lambda}|c_\lambda|^p e^{-p|\lambda|^2}<+\infty$$
there exists $f\in\calf^p$ such that $f|\Lambda=c$.

Also $\Lambda$ is an \textit{interpolating sequence for $\calf^\infty$} if for every sequence of values $c$ such that
$$\sup_{\lambda\in\Lambda}|c_\lambda| e^{-|\lambda|^2}<+\infty$$
there exists $f\in\calf^\infty$ such that $f|\Lambda=c$.
\end{defn}

Thus the interpolating sequences are the sequences such that the values functions from the space take on the sequence can be described purely in terms of
a natural growth condition. Seip and Wallst{\'e}n then proved the following result:
\begin{thm}[{\cite[Theorem 2.2, Theorem 2.4]{Se92},\cite[Theorem 1.2]{SeWa}}]
A sequence $\Lambda$ is an interpolating sequence for $\calf^p$, where $p\in[1,\infty]$, if and only if
\begin{itemize}
\item $\Lambda$ is a uniformly separated sequence, that is $\inf_{\lambda\neq\lambda'}|\lambda-\lambda'|>0$ and
\item The upper uniform density of $\Lambda$, $\mathcal{D}^+(\Lambda)=\displaystyle{\limsup_{r\to \infty}\sup_{z\in\C} \frac{\# \left(\Lambda\cap \overline{D(z,r)}\right)}
{\pi r^2}}<\frac{2}{\pi}$.
\end{itemize}
\end{thm}

There are many generalisations of this result, see \cite{MMO} and the references therein. We shall consider instead sequences $\Lambda\subset\C$ that are
uniformly separated but whose density is exactly the critical value, that is $\mathcal{D}^+(\Lambda)=2/\pi$. We shall not consider all such sequences but
instead restrict ourselves to those sequences that we have extra information about.

We consider first, as an instructive example, the integer lattice (suitably scaled)
$$\Lambda=\sqrt{\frac{\pi}{2}}(\mathbb{Z}+i\mathbb{Z}),$$
which is a sequence of critical density $2/\pi$. The Weierstrass $\sigma$-function associated to $\Lambda$ is defined by
$$\sigma(z)=z\prod_{\lambda\in\Lambda_0}\left(1-\frac{z}{\lambda}\right)e^{\frac{z}{\lambda}+\frac{1}{2}\frac{z^2}{\lambda^2}},$$
where we use the notation $\Lambda_{\lambda}=\Lambda\backslash\{\lambda\}$. Note that $\Lambda$ is the zero-set of $\sigma$, and
$$|\sigma(z)|\simeq e^{|z|^2}d(z,\Lambda)$$
for all $z\in\mathbb{C}$ \cite[p. 108]{SeWa}. Here $d$ refers to the usual distance between a point and a set.

Given a sequence $(a_{\lambda})_{\lambda\in\Lambda}$, we define the principal value of its sum to be
$$\text{p.v.}\sum_{\lambda\in\Lambda}a_{\lambda}=\lim_{R\to\infty} \sum_{|\lambda|<R}a_{\lambda}.$$
We are ready to state our main result, in this special case:
\begin{thm}\label{ClassF}
Let $\Lambda=\sqrt{\frac{\pi}{2}}(\mathbb{Z}+i\mathbb{Z})$. There exists $f\in\calf^1$  satisfying $f|\Lambda=c$ if and only if
\begin{itemize}
\item $\displaystyle{\sum_{\lambda\in\Lambda}|c_\lambda| e^{-|\lambda|^2}<+\infty}$,
\item $\displaystyle{\sum_{\lambda'\in\Lambda}\Bigl|\sum_{\lambda\in\Lambda_{\lambda'}}\frac{c_\lambda}{\sigma'(\lambda)(\lambda-\lambda')}\Bigr|<+\infty}$ and
\item $\displaystyle{\sum_{\lambda'\in\Lambda}\Bigl|\sum_{\lambda\in\Lambda_{\lambda'}}\frac{c_\lambda}{\sigma'(\lambda)(\lambda-\lambda')^2}\Bigr|<+\infty}$.
\end{itemize}
There exists $f\in\calf^p$ for $1<p<\infty$ satisfying $f|\Lambda=c$ if and only if
\begin{itemize}
\item $\displaystyle{\sum_{\lambda\in\Lambda}|c_\lambda|^p e^{-p|\lambda|^2}<+\infty}$ and
\item $\displaystyle{\sum_{\lambda'\in\Lambda}\Bigl|\textup{p.v.}\sum_{\lambda\in\Lambda_{\lambda'}}\frac{c_\lambda}{\sigma'(\lambda)(\lambda-\lambda')}\Bigr|^p<+\infty}$.
\end{itemize}
There exists $f\in\calf^\infty$ satisfying $f|\Lambda=c$ if and only if
\begin{itemize}
\item $\displaystyle{\sup_{\lambda\in\Lambda}|c_\lambda| e^{-|\lambda|^2}<+\infty}$,
\item
$\displaystyle{\sup_{\lambda'\in\Lambda_0}\Bigl|-\frac{c_0}{\sigma'(0)\lambda'}+\textup{p.v.}\sum_{\lambda\in\Lambda_{\lambda'}\backslash\{0\}}\frac{c_\lambda}{\sigma'(\lambda)}\bigl(\frac{1}{\lambda-\lambda'}-\frac{1}{\lambda}\bigr)\Bigr|<+\infty}$
and
\item $\displaystyle{\sup_{\lambda'\in\Lambda}\Bigl|\textup{p.v.}\sum_{\lambda\in\Lambda_{\lambda'}}\frac{c_\lambda}{\sigma'(\lambda)(\lambda-\lambda')^2}\Bigr|<+\infty}$.
\end{itemize}
\end{thm}

In fact these results hold for any sequence $\Lambda$ that is the zero set of a function $\tau$ with growth similar to the Weierstrass $\sigma$-function.
Specifically, suppose $\tau$ is an entire function such that
\begin{itemize}
\item The zero-sequence $\mathcal{Z}(\tau)$ of $\tau$  is uniformly separated,
\item $\sup_{z\in\mathbb{C}}d(z,\mathcal{Z}(\tau))<+\infty$ and
\item $|\tau(z)|\simeq e^{|z|^2}d(z,\mathcal{Z}(\tau))$ for all $z\in\mathbb{C}$.
\end{itemize}
Then Theorem~\ref{ClassF} holds if we replace $\sigma$ by $\tau$ and take $\Lambda=\mathcal{Z}(\tau)$. Such a set is always a set of critical density. The
existence of many such functions $\tau$ is guaranteed by Theorem~\ref{multiplier}.

Our work, both the results and the proofs, is inspired by a similar result due to Levin in the classical Paley-Wiener spaces \cite[Lecture 21]{Le}. In
these spaces, the integers are an interpolating sequence in almost every situation, however this fails in the two extremes, namely the $L^1$ and
$L^\infty$ cases. Levin completely described the traces of functions in the $L^\infty$ spaces on the integers, and Ber (see \cite[Lecture 21]{Le} and also
\cite{Ber}) solved the same problem in the $L^1$ case. While a discrete version of the Hilbert transform is the key ingredient in these results, we shall
see that it is discrete versions of the Cauchy and Beurling-Ahlfors transforms that shall play a similar role in the Fock context.

We shall in fact consider more general spaces, in which the function $|z|^2$ is replaced by a subharmonic function $\phi$ whose Laplacian $\Delta\phi$ is
a doubling measure.

The paper is structured as follows: In Section~2 we state some definitions and basic properties to be used later (namely of doubling measures, generalised
Fock spaces and generalised lattices). Section~3 contains the statements of our results. In Section~4 we prove two representation formulas for functions
in our generalised Fock spaces in terms of the values of the function on a critical lattice. In Section~5 we study a discrete version of the
Beurling-Ahlfors transform. Finally in Section~6 we prove the statements in Section~3.

We shall use the following standard notation: The expression  $f \lesssim g$ means that there is a constant $C$ independent of the relevant variables such
that $f\le C g$, and $f\simeq g$ means that $f\lesssim g$ and $g\lesssim f$.

\section{Technical Preliminaries}
This section contains technical results that we shall repeatedly use in our proofs, as well as a precise definition of the spaces we are studying. Much of
the development follows \cite{MMO}.
\subsection{Doubling Measures}
\begin{defn}
A nonnegative Borel measure $\mu$ in $\C$ is called \textit{doubling} if  there exists $C>0$ such that
$$\mu(D(z,2r))\leq C \mu(D(z,r))$$
for all $z\in\C$ and $r>0$. We denote by $C_\mu$ the infimum of the constants $C$ for which the inequality holds.
\end{defn}
Let $\phi$ be a (non-harmonic) subharmonic function whose Laplacian $\Delta\phi$ is a doubling measure. Canonical examples of such functions are given by
$\phi(z)=|z|^\gamma$ where $\gamma>0$. Writing $\mu=\Delta\phi$ we define, for $z\in\mathbb{C}$, $\rho_\phi(z)$ to be the radius such that
$\mu(D(z,\rho_\phi(z)))=1$. We shall normally ignore the dependence on $\phi$ and simply write $\rho(z)$.

We have the following estimates from \cite[p. 869]{MMO}: There exist $\eta>0$, $C_0>0$ and $\beta\in(0,1)$ such that
\begin{equation}\label{eqn5}
C_0^{-1}|z|^{-\eta}\leq\rho(z)\leq C_0|z|^\beta\text{ for }|z|>1.
\end{equation}
and
\begin{equation}\label{Lip}
|\rho(z)-\rho(\zeta)|\leq|z-\zeta|\text{ for }z,\zeta\in\mathbb{C}.
\end{equation}
So $\rho$ is a Lipschitz function, and so in particular is continuous. We will write
$$D^r(z)=D(z,r\rho(z))$$
and
$$D(z)=D^1(z).$$
We then have the following estimate, which we shall repeatedly make use of:
\begin{lemma}[{\cite[p. 205]{Ch}}]\label{Christ}
If $\zeta\not\in D(z)$ then
$$\frac{\rho(z)}{\rho(\zeta)}\lesssim\left(\frac{|z-\zeta|}{\rho(\zeta)}\right)^{1-t}$$
for some $t\in(0,1)$ depending only on the doubling constant, $C_\mu$.
\end{lemma}

We note, as in \cite{Ch}, that $\rho^{-2}$ can be seen as a regularisation of $\Delta\phi$. Then if we define $d_\phi$ to be the distance induced by the
metric $\rho(z)^{-2}dz\otimes d\overline{z}$ we have:
\begin{lemma}[{\cite[Lemma 4]{MMO}}]\label{lem4}
There exists $\delta>0$ such that for every $r>0$ there exists $C_r>0$ such that
\begin{itemize}
\item $\displaystyle{C_r^{-1}\frac{|z-\zeta|}{\rho(z)}\leq d_\phi(z,\zeta)\leq C_r\frac{|z-\zeta|}{\rho(z)}}\text{ if }|z-\zeta|\leq r\rho(z)$ and
\item $\displaystyle{C_r^{-1}\left(\frac{|z-\zeta|}{\rho(z)}\right)^\delta\leq d_\phi(z,\zeta)\leq C_r\left(\frac{|z-\zeta|}{\rho(z)}\right)^{2-\delta}}\text{ if }|z-\zeta|> r\rho(z)$.
\end{itemize}
\end{lemma}

\subsection{Generalised Fock spaces and interpolating sequences}
As before let $\phi$ be a subharmonic function whose Laplacian $\Delta\phi$ is a doubling measure. The following are all generalisations of the
development in the Introduction, with $\phi(z)$ playing the role of $|z|^2$. The generalised Fock spaces we deal with are defined as
$$\Fpa =\{f\in H(\mathbb{C}):\|f\|_\Fpa^p=\iC|f(z)|^p e^{-p\phi(z)}\frac{dm(z)}{\rho(z)^2}<+\infty\}\text{, for }1\leq p<+\infty$$
and
$$\Finfa =\{f\in H(\mathbb{C}):\|f\|_{\Finfa}=\sup_{z\in\C}|f(z)|e^{-\phi(z)}<+\infty\}.$$
We shall assume (see \cite[Theorem 14]{MMO}) that $\phi\in\mathcal{C}^\infty(\C)$. It is worth noting, as in~\cite[p. 863]{MMO}, that there are many
spaces of functions which correspond to $\Fpa$ for some $\phi$, although this may not be initially apparent. We shall take the following definitions
verbatim from \cite{MMO}:
\begin{defn}
A sequence $\Lambda\subseteq\mathbb{C}$ is an \textit{interpolating sequence for $\Fpa$, where $1\leq p<+\infty$} if for every sequence of values $c$ such
that
$$\sum_{\lambda\in\Lambda}|c_\lambda|^p e^{-p\phi(\lambda)}<+\infty$$
there exists $f\in\Fpa$ such that $f|\Lambda=c$.

Also $\Lambda$ is an \textit{interpolating sequence for $\Finfa$} if for every sequence of values $c$ such that
$$\sup_{\lambda\in\Lambda}|c_\lambda| e^{-\phi(\lambda)}<+\infty$$
there exists $f\in\Finfa$ such that $f|\Lambda=c$.
\end{defn}

\begin{defn}
A sequence $\Lambda$ is \textit{$\rho$-separated} if there exists $\delta>0$ such that
$$|\lambda-\lambda'|\geq\delta\max\{\rho(\lambda),\rho(\lambda')\}\text{, for }\lambda\neq\lambda'.$$
\end{defn}

One consequence of Lemma~\ref{lem4} is that a sequence $\Lambda$ is $\rho$-separated if and only if it is uniformly separated with respect to the distance
$d_\phi$, that is $\inf_{\lambda\neq\lambda'}d_\phi(\lambda,\lambda')>0$.
\begin{defn}
Assume that $\Lambda$ is a $\rho$-separated sequence and denote $\mu=\Delta\phi$ as before. The \textit{upper uniform density of $\Lambda$} with respect
to $\Delta\phi$ is
$$\mathcal{D}_{\Delta\phi}^+(\Lambda)=\limsup_{r\rightarrow\infty}\sup_{z\in\C}\frac{\#\left(\Lambda\bigcap\overline{D(z,r\rho(z))}\right)}{\mu(D(z,r\rho(z)))}.$$
\end{defn}

It should be noted that replacing $\phi(z)$ by $|z|^2$ in this definition does not produce the density given in Theorem~\ref{ClassF}, but rather a
constant multiple of it. We have:
\begin{thm}[{\cite[Theorem B]{MMO}}]
A sequence $\Lambda$ is interpolating for $\Fpa$, where $p\in[1,\infty]$, if and only if $\Lambda$ is $\rho$-separated and
$\mathcal{D}_{\Delta\phi}^+(\Lambda)<\frac{1}{2\pi}$.
\end{thm}

We finish with a Plancherel-Polya type inequality:
\begin{lemma}[{\cite[Lemma 19(a)]{MMO}}]\label{PlaPol}
Let $1\leq p<\infty$. For any $r>0$ there exists $C=C(r)>0$ such that for any $f\in H(\mathbb{C})$ and $z\in\mathbb{C}$
$$|f(z)|^pe^{-p\phi(z)}\leq C\int_{D^r(z)}|f(\zeta)|^pe^{-p\phi(\zeta)}\frac{dm(\zeta)}{\rho(\zeta)^2}.$$
\end{lemma}

This has an elementary but useful consequence. Suppose that $\Lambda$ is a $\rho$-separated sequence and that $f\in\Fpa$. Then
\begin{equation}\label{discPlaPol}
\sum_{\lambda\in\Lambda}|f(\lambda)|^pe^{-p\phi(\lambda)}\leq
C\sum_{\lambda\in\Lambda}\int_{D^{\delta/2}(\lambda)}|f(\zeta)|^pe^{-p\phi(\zeta)}\frac{dm(\zeta)}{\rho(\zeta)^2}\leq
C\iC|f(\zeta)|^pe^{-p\phi(\zeta)}\frac{dm(\zeta)}{\rho(\zeta)^2}<\infty
\end{equation}
where $\delta$ is the constant appearing in the definition of $\rho$-separation and $C=C(\delta/2)$.

Moreover, if $f\in\Fpa$ for $1\leq p<\infty$ then
$$|f(z)|^pe^{-p\phi(z)}\rightarrow0$$
uniformly as $|z|\rightarrow\infty$, from which we infer that $\Fpa\subseteq\Finfa$.

\subsection{Generalised lattices}
We shall now consider analogues of the integer lattice considered earlier, that play a similar role in our generalised spaces.

\begin{thm}[{\cite[Theorem 17]{MMO}}]\label{multiplier}
Let $\phi$ be a subharmonic function such that $\Delta\phi$ is a doubling measure. There exists an entire function $g$ such that
\begin{itemize}
\item The zero-sequence $\mathcal{Z}(g)$ of $g$ is $\rho_\phi$-separated and $\displaystyle \sup\limits_{z\in\C}d_\phi(z, \mathcal Z(g))<\infty$.
\item $|g(z)|\simeq e^{\phi(z)}d_\phi(z,\mathcal{Z}(g))$ for all $z\in\C$.
\end{itemize}
The function $g$ can be chosen so that, moreover, it vanishes on a prescribed $z_0\in\C$. We say that $g$ is a multiplier associated to $\phi$.
\end{thm}

Furthermore~\cite[Lemma 37]{MMO} shows that $\mathcal{D}_{\Delta\phi}^+(\mathcal{Z}(g))=1/2\pi$. We shall now regard $\phi$ and $g$ (and consequently
$\rho$) as fixed and we will say that $\Lambda=\mathcal{Z}(g)$ is a \textit{critical lattice associated to the multiplier $g$}. The multiplier can be
thought of as playing the same role in Fock spaces that sine-type functions play in Paley-Wiener spaces.

Suppose now that $f\in\Fpa$, that $z$ is uniformly bounded away from $\Lambda$ in the distance $d_\phi$ and $\epsilon>0$ is arbitrary. Then
$$\left|\frac{f(z)}{g(z)}\right|^p\simeq|f(z)|^p e^{-p\phi(z)}<\epsilon$$
uniformly as $|z|\rightarrow\infty$, where we have used Theorem~\ref{multiplier} and Lemma~\ref{PlaPol}. In fact, if $z_n$ is any $\rho$-separated
sequence that satisfies $d_\phi(z_n,\Lambda)\geq C>0$ for all $n$ (here $C$ is any positive constant), then Theorem~\ref{multiplier} implies that
\begin{equation}\label{rhosepseries}
\sum_n\Bigl|\frac{f(z_n)}{g(z_n)}\Bigr|^p<\infty.
\end{equation}

For any $\lambda\in\Lambda$ Theorem~\ref{multiplier} and Lemma~\ref{lem4} show that $|g'(\lambda)|\simeq e^{\phi(\lambda)}/\rho(\lambda)$ and we conclude
that
\begin{equation}\label{basicnec}
\sum_{\lambda\in\Lambda}\Bigl|\frac{f(\lambda)}{g'(\lambda)\rho(\lambda)}\Bigr|^p<+\infty
\end{equation}
by invoking \eqref{discPlaPol}.

There exists $\delta_1>0$ such that $|\lambda-\lambda'|>2\delta_1\max\{\rho(\lambda),\rho(\lambda')\}$ for all $\lambda\neq\lambda'$. Recall that
$D^r(z)=D(z,r\rho(z))$ and $D(z)=D^1(z).$ We will write
$$Q_\lambda=\{z\in\mathbb{C}:d_\phi(z,\Lambda)=d_\phi(z,\lambda)\}$$
for $\lambda\in\Lambda$. Lemma~\ref{lem4} implies that for any $0<\delta\leq\delta_1$ then $D^{\delta}(\lambda)\subseteq Q_\lambda$ and for some constant
$R_1>0$ we have $Q_\lambda\subseteq D^{R_1}(\lambda)$. In fact the sets $D^{\delta_1}(\lambda)$ are pairwise disjoint and
$\mathbb{C}=\bigcup_{\lambda\in\Lambda}Q_\lambda$. Additionally we have
\begin{equation}\label{intQ}
\int_{Q_{\lambda}}\frac{dm(z)}{\rho(z)^2}\simeq\frac{1}{\rho(\lambda)^2}\int_{Q_{\lambda}}dm(z)=\frac{|Q_{\lambda}|}{\rho(\lambda)^2}
\leq\frac{|D^{R_1}(\lambda)|}{\rho(\lambda)^2}=\pi R_1^2
\end{equation}
where $|A|$ is the Lebesgue measure of the set $A$.

We shall henceforth assume that $0\in\Lambda$. This can always be achieved by fixing some $\lambda_0\in\Lambda$ and translating this point to the origin.
This is merely a matter of convenience and will simplify many of our calculations.

Let $\beta$ and $\eta$ be as in~\eqref{eqn5} and chose $\alpha>2+2\eta$. Then, for $0<\delta<\delta_1$,
$$\sum_{|\lambda|>1}\frac{1}{|\lambda|^\alpha}\lesssim\sum_{|\lambda|>1}\frac{\rho(\lambda)^2}{|\lambda|^{\alpha-2\eta}}
\simeq\sum_{|\lambda|>1}\int_{D^\delta(\lambda)}\frac{1}{|z|^{\alpha-2\eta}}\leq\int_{\mathbb{C}\backslash D^\delta(0)}\frac{1}{|z|^{\alpha-2\eta}}
<+\infty$$ so that $\sum_{\lambda\in\Lambda_0}\lambda^{-\alpha}$ is an absolutely convergent sum.

\subsection{Discrete potentials}
In this section we shall only assume that $\Lambda$ is $\rho$-separated, although when we apply it later we shall take $\Lambda=\mathcal{Z}(g)$. Given a
sequence $(d_{\lambda})_{\lambda\in\Lambda}$ such that $\left(d_\lambda\rho(\lambda)^\alpha\right)_{\lambda\in\Lambda}\in\ell^{p}$ where $\alpha$ is real,
we will say that $d\in\ell^{p}(\rho^{\alpha})$. We shall repeatedly need the following result:
\begin{lemma}\label{usef}
(i) If $\Lambda$ is $\rho$-separated, $1\leq p\leq2$ and
$d\in\ell^{p}(\rho^{-1})$ then
$$\sum_{\lambda\in\Lambda_{\lambda'}}\frac{d_{\lambda}}{|\lambda'-\lambda|^3}\in\ell^{p}(\rho^{2}).$$
(ii) If $\Lambda$ is $\rho$-separated, $1\leq p\leq+\infty$ and
$d\in\ell^{p}(\rho^{-1})$ then
$$\sum_{\lambda\in\Lambda_{\lambda'}}\frac{d_{\lambda}}{|\lambda'-\lambda|^{N+1}}\in\ell^{p}(\rho^{N})$$
for any integer $N>1/t$ where $t$ is the constant occurring in Lemma~\ref{Christ}.
\end{lemma}
\begin{proof}
(i) Define $\tilde{d}_\lambda=d_\lambda/\rho(\lambda)$ and
$$L_{\lambda'}(\tilde{d})=\sum_{\lambda\in\Lambda_{\lambda'}}\frac{\tilde{d}_{\lambda}\rho(\lambda)\rho(\lambda')^2}{|\lambda'-\lambda|^3}.$$
H\"{o}lder's inequality and Lemma~\ref{Christ} show that this sum converges, and it clearly gives rise to a linear operator on $\ell^{p}$. We will show
that it is in fact a bounded operator from $\ell^{p}$ to $\ell^{p}$, which will imply the claimed result. Note first that
$$\sum_{\lambda'\in\Lambda}|L_{\lambda'}|\leq\sum_{\lambda\in\Lambda}|\tilde{d}_\lambda|\rho(\lambda)\sum_{\lambda'\in\Lambda_{\lambda}}\frac{\rho(\lambda')^2}{|\lambda'-\lambda|^3}\lesssim\sum_{\lambda\in\Lambda}|\tilde{d}_\lambda|$$
so that $L$ is a bounded linear operator from $\ell^{1}$ to $\ell^{1}$. Here we have used the fact that
$$\sum_{\lambda'\in\Lambda_{\lambda}}\frac{\rho(\lambda')^2}{|\lambda'-\lambda|^3}\simeq\sum_{\lambda'\in\Lambda_{\lambda}}\int_{D^\delta(\lambda')}\frac{dm(z)}{|z-\lambda|^3}\leq\int_{\mathbb{C}\backslash D^\delta(\lambda)}\frac{dm(z)}{|z-\lambda|^3}\simeq\rho(\lambda)^{-1}$$
where $0<\delta<\delta_1$.

We now show that $L$ is a bounded operator from $\ell^{2}$ to $\ell^{2}$, using Schur's test (see eg. \cite{Wi}). We consider $L$ as an integral operator
with kernel $K(\lambda',\lambda)=\frac{\rho(\lambda)\rho(\lambda')^2}{|\lambda'-\lambda|^3}$ for $\lambda\neq\lambda'$ and $K(\lambda,\lambda)=0$. Now
\begin{equation}\label{sch1}
\sum_{\lambda\in\Lambda}K(\lambda',\lambda)\rho(\lambda)=\rho(\lambda')^2\sum_{\lambda\in\Lambda_{\lambda'}}\frac{\rho(\lambda)^2}{|\lambda'-\lambda|^3}\lesssim\rho(\lambda')
\end{equation}
and
$$\sum_{\lambda'\in\Lambda}K(\lambda',\lambda)\rho(\lambda')=\rho(\lambda)\sum_{\lambda'\in\Lambda_\lambda}\frac{\rho(\lambda')^3}{|\lambda'-\lambda|^3}\lesssim\rho(\lambda)\int_{\mathbb{C}\backslash D(\lambda')}\frac{\rho(z)}{|\lambda'-z|^3}dm(z).$$
Applying Lemma~\ref{Christ} we have
$$\int_{\mathbb{C}\backslash D(\lambda')}\frac{\rho(z)}{|\lambda'-z|^3}dm(z)\lesssim\rho(\lambda')^t\int_{\mathbb{C}\backslash D(\lambda')}\frac{dm(z)}{|\lambda'-z|^{2+t}}\simeq1$$
so that
\begin{equation}\label{sch2}
\sum_{\lambda'\in\Lambda}K(\lambda',\lambda)\rho(\lambda')\lesssim\rho(\lambda).
\end{equation}
Combining now \eqref{sch1} and \eqref{sch2} and applying the Schur test shows that $L$ is indeed bounded from $\ell^{2}$ to $\ell^{2}$. Applying now the
Riesz-Thorin interpolation theorem (see, eg. \cite[Chap. 6, \S 5, Thm (6.27)]{Fo}) completes the proof.

(ii) We use the same notation. Define
$$M_{\lambda'}(\tilde{d})=\sum_{\lambda\in\Lambda_{\lambda'}}\frac{\tilde{d}_{\lambda}\rho(\lambda)\rho(\lambda')^N}{|\lambda'-\lambda|^{N+1}}.$$
Since $\Lambda$ is $\rho$-separated and $N\geq2$ we have
$$\frac{\rho(\lambda')^N}{|\lambda'-\lambda|^{N+1}}\lesssim\frac{\rho(\lambda')^2}{|\lambda'-\lambda|^{3}}$$
so that (i) shows that $M$ defines a bounded linear operator from $\ell^{1}$ to $\ell^{1}$.

Also
$$\sup_{\lambda'\in\Lambda}\left|M_{\lambda'}\right|\leq\|\tilde{d}\|_{\ell^{\infty}}\sup_{\lambda'\in\Lambda}\rho(\lambda')^N\sum_{\lambda\in\Lambda_{\lambda'}}\frac{\rho(\lambda)}{|\lambda'-\lambda|^{N+1}}.$$
Applying again Lemma~\ref{Christ} we have
$$\sum_{\lambda\in\Lambda_{\lambda'}}\frac{\rho(\lambda)}{|\lambda'-\lambda|^{N+1}}\lesssim\int_{\mathbb{C}\backslash D(\lambda')}\frac{dm(z)}{|\lambda'-z|^{N+1}\rho(z)}\lesssim\rho(\lambda')^{-\frac{1}{t}}\int_{\mathbb{C}\backslash
D(\lambda')}\frac{dm(z)}{|\lambda'-z|^{2+N-\frac{1}{t}}}\simeq\rho(\lambda')^{-N}$$ since $N>1/t$. Consequently
$$\sup_{\lambda'\in\Lambda}\left|M_{\lambda'}\right|\lesssim\|\tilde{d}\|_{\ell^{\infty}}$$ so that $M$ defines a bounded linear operator from $\ell^{\infty}$ to
$\ell^{\infty}$. Once more the Riesz-Thorin interpolation theorem completes the proof.
\end{proof}

\section{Statement of our main results.} We are ready to state our results, in full generality. As before $\Lambda$ is a critical lattice, the zero
sequence of $g$ a multiplier associated to $\phi$. We begin with the simplest case, which is the Hilbert space $\Fdosa$, where we need only slightly
modify Theorem~\ref{ClassF}:
\begin{thm}\label{main2}
Let $\Lambda$ be a critical lattice associated to the multiplier $g$. There exists $f\in\Fdosa$ satisfying $f|\Lambda=c$ if and only if
\begin{itemize}
\item $\displaystyle{\sum_{\lambda\in\Lambda}|c_\lambda|^2 e^{-2\phi(\lambda)}<+\infty}$ and
\item $\displaystyle{\sum_{\lambda'\in\Lambda}\Bigl|\textup{p.v.}\sum_{\lambda\in\Lambda_{\lambda'}}\frac{c_\lambda}{g'(\lambda)(\lambda-\lambda')}\Bigr|^2 <+\infty}$.
\end{itemize}
\end{thm}

Our result in $\calf_\phi^1$ is also only a slight modification of Theorem~\ref{ClassF}.
\begin{thm}\label{main1}
Let $\Lambda$ be a critical lattice associated to the multiplier $g$. There exists $f\in\calf_\phi^1$ satisfying $f|\Lambda=c$ if and only if
\begin{itemize}
\item $\displaystyle{\sum_{\lambda\in\Lambda}|c_\lambda| e^{-\phi(\lambda)}<+\infty}$ and
\item $\displaystyle{\sum_{\lambda'\in\Lambda}\Bigl|\sum_{\lambda\in\Lambda_{\lambda'}}\frac{c_\lambda}{g'(\lambda)(\lambda-\lambda')}\Bigr|<+\infty}$ and
\item
$\displaystyle{\sum_{\lambda'\in\Lambda}\Bigl|\rho(\lambda')\sum_{\lambda\in\Lambda_{\lambda'}}\frac{c_\lambda}{g'(\lambda)(\lambda-\lambda')^2}\Bigr|<+\infty}$
\end{itemize}
\end{thm}
In general, the situation is slightly more complicated. We begin with the case $1<p<2$. Here there are two possibilities, depending on whether or not
$\rho^{p-2}$ is a Muckenhoupt $A_p$ weight (see Section~\ref{BeuAhlsecn} for the definition). If this additional assumption holds then our result is
essentially the same as in the classical case, otherwise we add an additional condition to our result. We also show in Section~\ref{BeuAhlsecn} that both
of these possibilities can occur.
\begin{thm}\label{mainbetw1n2}
Let $\Lambda$ be a critical lattice associated to the multiplier $g$ and suppose $1<p<2$.
\begin{itemize}
\item
If $\rho_\phi^{p-2}$ is an $A_p$ weight then there exists $f\in\Fpa$ satisfying $f|\Lambda=c$ if and only if
\begin{itemize}
\item[(a)] $\displaystyle{\sum_{\lambda\in\Lambda}|c_\lambda|^p e^{-p\phi(\lambda)}<+\infty}$ and
\item[(b)] $\displaystyle{\sum_{\lambda'\in\Lambda}\Bigl|\textup{p.v.}\sum_{\lambda\in\Lambda_{\lambda'}}\frac{c_\lambda}{g'(\lambda)(\lambda-\lambda')}\Bigr|^p<+\infty}$.
\end{itemize}
\item
If $\rho_\phi^{p-2}$ is not an $A_p$ weight then there exists $f\in\Fpa$ satisfying $f|\Lambda=c$ if and only if (a) and (b) hold, and in addition
\begin{itemize}
\item[(c)] $\displaystyle{\sum_{\lambda'\in\Lambda}\Bigl|\rho(\lambda')\textup{p.v.}\sum_{\lambda\in\Lambda_{\lambda'}}\frac{c_\lambda}{g'(\lambda)(\lambda-\lambda')^2}\Bigr|^p<+\infty}$.
\end{itemize}
\end{itemize}
\end{thm}
If $2<p<\infty$ then our result becomes more complicated, depending on the doubling constant.
\begin{thm}\label{mainbig2}
Let $\Lambda$ be a critical lattice associated to the multiplier $g$ and suppose $2<p<\infty$. Let $t$ be the constant occurring in Lemma~\ref{Christ}
(which depends on the doubling constant).
\begin{rlist}
\item
If $t>1/2$ and $\rho_\phi^{p-2}$ is an $A_p$ weight then there exists $f\in\Fpa$ satisfying $f|\Lambda=c$ if and only if (a) and (b) hold.
\item
If $t>1/2$ and $\rho_\phi^{p-2}$ is not an $A_p$ weight then there exists $f\in\Fpa$ satisfying $f|\Lambda=c$ if and only if (a), (b) and (c) hold.
\item If $t\leq1/2$ then there exists $f\in\Fpa$ satisfying $f|\Lambda=c$ if and only if (a) holds and
\begin{itemize}
\item[(b$'$)] There exists an integer $N>\frac{1}{t}$ such that, for every $1\leq n\leq N$,
$$\sum_{\lambda'\in\Lambda}\Bigl|\rho(\lambda')^{n-1}\textup{p.v.}\sum_{\lambda\in\Lambda_{\lambda'}}\frac{c_\lambda}{g'(\lambda)(\lambda-\lambda')^n}\Bigr|^p<+\infty.$$
\end{itemize}
\end{rlist}
\end{thm}
We finally state our result for $\Finfa$, which again depends on the doubling constant.
\begin{thm}\label{maininf}
Let $\Lambda$ be a critical lattice associated to the multiplier $g$ and let $t$ be the constant occurring in Lemma~\ref{Christ}. There exists
$f\in\Finfa$ satisfying $f|\Lambda=c$ if and only if
\begin{itemize}
\item $\displaystyle{\sup_{\lambda\in\Lambda}|c_\lambda| e^{-\phi(\lambda)}<+\infty}$,
\item
$\displaystyle{\sup_{\lambda'\in\Lambda_0}\Bigl|-\frac{c_0}{g'(0)\lambda'}+\textup{p.v.}\sum_{\lambda\in\Lambda_{\lambda'}\backslash\{0\}}\frac{c_\lambda}{g'(\lambda)}\bigl(\frac{1}{\lambda-\lambda'}-\frac{1}{\lambda}\bigr)\Bigr|<+\infty}$
and
\item There exists an integer $N>\frac{1}{t}$ such that, for every $2\leq n\leq N$,
$$\sup_{\lambda'\in\Lambda}\Bigl|\rho(\lambda')^{n-1}\textup{p.v.}\sum_{\lambda\in\Lambda_{\lambda'}}\frac{c_\lambda}{g'(\lambda)(\lambda-\lambda')^n}\Bigr|<+\infty.$$
\end{itemize}
\end{thm}
We finish this section by showing that the cases $t\leq1/2$ and $t>1/2$ are both possible. In the classical example of $\phi(z)=|z|^2$ we may take $t$ to
be arbitrarily close to $1$. For the other case we consider the function $\phi(z)=|z|^\gamma$ where $\gamma>0$. We may assume, by normalising
appropriately, that $\rho(0)=1$. Then for $\zeta\not\in D(0)$ we have $\rho(\zeta)\simeq|\zeta|^{1-\gamma/2}$. Taking now $z=0$ in Lemma~\ref{Christ} we
see that we must have
$$|\zeta|^{1-\gamma/2}\lesssim|\zeta|^{1-t}$$
for $\zeta\not\in D(0)$ so that $t\leq\gamma/2$. Thus, for the function $\phi(z)=|z|^\gamma$ when $\gamma\leq1$, we must have $t\leq1/2$.

\section{Representation formulas}
In this section we will prove two representation formulas for functions in our generalised Fock spaces in terms of the values of the function on a
critical lattice. These formulas are reminiscent of the Lagrange interpolation formula.
\begin{lemma}\label{lemrepinf}
Let $\Lambda$ be a critical lattice associated to the multiplier $g$. If $f\in\Finfa$ then
\begin{equation}\label{repinf}
f(z)=g(z)\Bigl[w_0+\frac{f(0)}{g'(0)z}+\text{\textup{p.v.}}
\sum_{\lambda\in\Lambda_0}\frac{f(\lambda)}{g'(\lambda)}\left(\frac{1}{z-\lambda}+\frac{1}{\lambda}\right)\Bigr]
\end{equation}
where $w_0=\lim_{z\rightarrow0}\frac{d}{dz}(\frac{zf(z)}{g(z)})=\frac{f'(0)}{g'(0)}-\frac{g''(0)}{2g'(0)}$.
\end{lemma}

\begin{proof}
We denote $d_{\lambda}=f(\lambda)/g'(\lambda)$ and note that
$|d_\lambda/\rho(\lambda)|\simeq|c_\lambda|e^{-\phi(\lambda)}\leq\sup_{z\in\mathbb{C}}|f(z)|e^{-\phi(z)}=\|f\|_\Finfa$ so that
$\left(d_\lambda/\rho(\lambda)\right)_{\lambda\in\Lambda}\in\ell^{\infty}$ and $\left\|d_\lambda/\rho(\lambda)\right\|_\infty\lesssim\|f\|_{\Finfa}$. Let
$\beta$ and $\eta$ be as in~\eqref{eqn5} and fix a positive integer $n>2+2\eta+\beta$. We will write
\begin{equation}\label{lagr}
f(z)=\sum_{\lambda\in\Lambda}f(\lambda)g_\lambda(z)
\end{equation}
where $g_\lambda$ are entire functions satisfying
$g_\lambda(\lambda')=\delta_{\lambda\lambda'}$. The obvious
candidate for $g_\lambda(z)$ is the function
$g_\lambda(z)=\frac{g(z)}{g'(\lambda)(z-\lambda)}$, however the
resultant series is in general not convergent. We shall keep
$g_0(z)=\frac{g(z)}{g'(0)z}$, but instead take
$g_\lambda(z)=\frac{g(z)}{g'(\lambda)}\left(\frac{1}{z-\lambda}-p_{n-1}(z)\right)$
for $\lambda\neq0$, where $p_{n-1}$ is the Taylor polynomial of
degree $n-1$ of the function $C_\lambda(z)=\frac{1}{z-\lambda}$
expanded around $0$. Note that we still have
$g_\lambda(\lambda')=\delta_{\lambda\lambda'}$ but now the series is
pointwise convergent. In fact
\begin{equation}\label{tayl}
\frac{1}{z-\lambda}-p_{n-1}(z)=\frac{1}{z-\lambda}+\frac{1}{\lambda}+\frac{z}{\lambda^2}+\cdots+\frac{z^{n-1}}{\lambda^n}=\frac{z^n}{\lambda^n(z-\lambda)}
\end{equation}
so that if we define
$$G(z)=\frac{d_0}{z}+
\sum_{\lambda\in\Lambda_0}d_\lambda\left(\frac{1}{z-\lambda}+\frac{1}{\lambda}+\frac{z}{\lambda^2}+\cdots+\frac{z^{n-1}}{\lambda^n}\right) =\frac{d_0}{z}+
\sum_{\lambda\in\Lambda_0}d_\lambda\frac{z^n}{\lambda^n(z-\lambda)}$$ then for any $K$ a compact subset of $\mathbb{C}\backslash\Lambda$ we have
$$\sum_{\lambda\in\Lambda_0}\Bigl|d_\lambda\frac{z^n}{\lambda^n(z-\lambda)}\Bigr|\lesssim\sum_{\lambda\in\Lambda_0}\frac{|d_\lambda|}{|\lambda|^n} \lesssim\sum_{\lambda\in\Lambda_0}\frac{\rho(\lambda)}{|\lambda|^n}
\lesssim\sum_{\lambda\in\Lambda_0}\frac{1}{|\lambda|^{n-\beta}}<+\infty$$ for all $z\in K$, since $n-\beta>2+2\eta$. Hence $G$ defines a meromorphic
function on $\mathbb{C}$ with a simple pole at each $\lambda\in\Lambda$. Consequently $gG$ is an entire function that agrees with $f$ at each
$\lambda\in\Lambda$. This implies that there exists an entire function $h$ such that $f-gG=gh$.

Fix $\epsilon>0$ and $0<\delta<\delta_1$, and define $\Omega=\C\backslash\bigcup_{\lambda\in\Lambda}D^\delta(\lambda)$. Now for each $z\in\Omega$ we have
$$\left|\frac{f(z)}{g(z)}\right|\simeq\frac{|f(z)|e^{-\phi(z)}}{d_\phi(z,\Lambda)}\simeq|f(z)|e^{-\phi(z)}\leq\|f\|_\Finfa.$$
Also, when $z\in\Omega$, we obviously have
$$|G(z)|\leq\frac{|d_0|}{|z|}+|z|^n\sum_{\lambda\in\Lambda_0}\Bigl|\frac{d_{\lambda}}{\lambda^n(z-\lambda)}\Bigr|.$$
We split this sum over two separate ranges. For any $R>1$,
$$\sum_{|\lambda|>R}\Bigl|\frac{d_{\lambda}}{\lambda^n(z-\lambda)}\Bigr|\lesssim\sum_{|\lambda|>R}\frac{\rho(\lambda)}{|\lambda|^n|z-\lambda|}\lesssim\sum_{|\lambda|>R}\frac{1}{|\lambda|^{n}}<\epsilon$$
for sufficiently large $R$. Fixing one such $R$ we then have, for $|z|>2R$,
$$\sum_{0<|\lambda|\leq R}\left|\frac{d_{\lambda}}{\lambda^n(z-\lambda)}\right|\leq\frac{2}{|z|}\sum_{0<|\lambda|\leq R}\frac{|d_{\lambda}|}{|\lambda|^n}=\frac{C}{|z|}$$
for some constant $C$. Hence
$$|G(z)|\leq\left|\frac{d_0}{z}\right|+|z|^n\Bigl(\frac{C}{|z|}+\epsilon\Bigr)=o(z^n)$$
for $|z|\geq2R$. Gathering these estimates we have $|h(z)|=o(z^n)$
for $z\in\Omega$ of sufficiently large modulus. Applying now the
maximum principle to $h$ on $D^\delta(\lambda)$ for each
$\lambda\in\Lambda$ far from the origin we see that this holds for
all $z\in\mathbb{C}$ of sufficiently large modulus. We conclude that
$h$ is a polynomial of degree less than or equal to $n-1$.

Note that if we define
$$H(z)=G(z)-\frac{d_0}{z}=\sum_{\lambda\in\Lambda_0}d_\lambda\left(\frac{1}{z-\lambda}+\frac{1}{\lambda}+\frac{z}{\lambda^2}+\cdots+\frac{z^{n-1}}{\lambda^n}\right)$$
then $H^{(j)}(0)=0$ for $0\leq j<n$. Since
$$h(z)=\frac{f(z)}{g(z)}-G(z)=\frac{1}{z}\left(\frac{zf(z)}{g(z)}-d_0\right)-H(z)$$
we may evaluate $h$ by computing the Laurent expansion of $f/g$ around $0$. This yields
\begin{equation}\label{h}
h(z)=\sum_{m=1}^n\frac{1}{m!}\lim_{w\rightarrow0}\frac{d^m}{dw^m}\left(\frac{wf(w)}{g(w)}\right)z^{m-1}.
\end{equation}
Fix some $0<\delta'<\delta_1$ and define $\gamma(R)$ to be the closed curve consisting of the portion of the circle $|z|=R$ for which $|z-\lambda|\geq
\delta'\rho(\lambda)$ and of the portions the circles $|z-\lambda|=\delta'\rho(\lambda)$ that intersect the circle $|z|=R$ in such a manner that $\lambda$
is in the domain bounded by $\gamma(R)$ if and only if $|\lambda|<R$. Then the Cauchy residue theorem implies that
$$\frac{1}{2\pi i}\int_{\gamma(R)}\frac{f(w)}{g(w)w^m}dw=\frac{1}{m!}\lim_{w\rightarrow0}\frac{d^m}{dw^m}\left(\frac{wf(w)}{g(w)}\right)+\sum_{0<|\lambda|<R}\frac{d_\lambda}{\lambda^m}.$$
Now the length of the contour of integration is comparable to the length of the circle of radius $R$. Moreover $d_\phi(z,\Lambda)$ is bounded away from
$0$ for $z\in\gamma(R)$ so that $\left|f(z)/g(z)\right|$ is bounded above. This implies that
$$\lim_{R\rightarrow\infty}\frac{1}{2\pi i}\int_{\gamma(R)}\frac{f(w)}{g(w)w^m}dz=0$$
for $m\geq2$, whence
$$\frac{1}{m!}\lim_{w\rightarrow0}\frac{d^m}{dw^m}\left(\frac{wf(w)}{g(w)}\right)=-\text{\textup{p.v.}}\sum_{\lambda\in\Lambda_0}\frac{d_\lambda}{\lambda^m}.$$
Inserting this expression into~\eqref{h} yields
$$h(z)=\lim_{w\rightarrow0}\frac{d}{dw}\left(\frac{wf(w)}{g(w)}\right)-\text{\textup{p.v.}}\sum_{\lambda\in\Lambda_0}d_\lambda\sum_{m=2}^n\frac{z^{m-1}}{\lambda^m}.$$
Computing now $f=g(G+h)$ completes the proof.
\end{proof}
\begin{rem}
Given any function $f\in\Finfa$ the function $f+Cg$ is also in $\Finfa$ for any constant $C$, and the functions agree at every $\lambda\in\Lambda$. Thus
$\Lambda$ is not a set of uniqueness for this space. This result, however, tells us that this is the only possibility, that is if $f,\tilde{f}\in\Finfa$
and $f(\lambda)=\tilde{f}(\lambda)$ for all $\lambda\in\Lambda$ then $f-\tilde{f}=Cg$ for some constant $C$.
\end{rem}

\begin{cor}
Let $\Lambda$ be a critical lattice associated to the multiplier $g$. If $f\in\Fpa$ for $1\leq p<+\infty$ then
\begin{equation}\label{repfinite}
f(z)=g(z)\text{\textup{p.v.}}\sum_{\lambda\in\Lambda}\frac{f(\lambda)}{g'(\lambda)(z-\lambda)}.
\end{equation}
\end{cor}
\begin{proof}
We use the same notation. Since $\Fpa\subseteq\Finfa$ we know that~\eqref{repinf} must hold. But now
$$\lim_{R\rightarrow\infty}\frac{1}{2\pi i}\int_{\gamma(R)}\frac{f(w)}{g(w)w}dz=0$$
so that $w_0=-\text{\textup{p.v.}}\sum_{\lambda\in\Lambda_0}\frac{d_\lambda}{\lambda}$.
\end{proof}
\begin{rems}
This shows that $\Lambda$ is a set of uniqueness for these spaces.

This representation is \eqref{lagr} with the obvious choice of $g_\lambda$, except that we are taking principal values of the sum. In fact if $p=1$ then
the sum appearing in~\eqref{repfinite} is absolutely convergent, so the principal value may be ignored. In this case the proof may be simplified by taking
$G$ to be this sum and estimating similarly. The decay of this function away from the lattice means we have no need to invoke the Cauchy residue theorem,
or involve principal values.
\end{rems}

\section{The discrete Beurling-Ahlfors transform}\label{BeuAhlsecn}
It is well known that the Beurling-Ahlfors transform given by
\begin{equation}\label{cnts}
T[f](\zeta)=\lim_{\epsilon\rightarrow0}\int_{\mathbb{C}\backslash D(\zeta,\epsilon)}\frac{f(z)}{(\zeta-z)^{2}}dm(z),
\end{equation}
where $m$ denotes the Lebesgue measure on the plane, is a bounded linear operator from $L^{p}(\mathbb{C})$ to $L^{p}(\mathbb{C})$ for $1<p<+\infty$. (It
should be noted that this differs from the usual definition by a factor of $-\frac{1}{\pi}$, and that it is customary to denote this limit as a principle
value. We have avoided doing so to eliminate any possible confusion with the principal value of a sum.) In fact this also holds if we replace
$L^{p}(\mathbb{C})$ by a more general weighted space. We make use of the following definition:
\begin{defn}\cite[Ch.V \S1 p. 194]{St93}
A weight $\omega$ on $\R^n$ is said to be a \textit{Muckenhoupt
$A_p$ weight} if it is locally integrable and there exists some
constant $A$ such that
\begin{equation}\label{Apdef}
\left(\frac{1}{|B|}\int_B \omega(x)dm(x)\right)\left(\frac{1}{|B|}\int_B \omega(x)^{-\frac{q}{p}}dm(x)\right)^{\frac{p}{q}}\leq A<\infty
\end{equation}
for all balls $B$ in $\R^n$. Here $m$ is Lebesgue measure on $\R^n$, $q$ is the H\"{o}lder conjugate exponent of $p$ (that is $\frac{1}{p}+\frac{1}{q}=1$)
and $|B|$ is the Lebesgue measure of the ball $B$. The least constant $A$ for which this holds is called the $A_p$ bound of $\omega$, denoted
$A_p(\omega)$
\end{defn}
We shall of course be interested in $\R^2$ which we identify with $\C$. Now the corollary to \cite[Ch. V,\S4.2, Th. 2]{St93} combined with \cite[Ch.
V,\S4.5.2]{St93} show that $T$ is a bounded linear operator from $L^{p}(\omega)$ to $L^{p}(\omega)$ for $1<p<+\infty$ for any $A_p$ weight $\omega$. (In
fact the proof is given for a much more general class of integral operators, of which $T$ is a special case.) We aim to use this property to study a
discrete analogue.

We shall be interested in the case when $\rho^{p-2}$ is an $A_p$ weight. Substituting into \eqref{Apdef} and re-formulating shows that this is equivalent
to saying that there exists some constant $A$ such that
\begin{equation}\label{Appart}
\frac{1}{|D|}\left(\int_D \rho(z)^p d\nu(z)\right)^\frac{1}{p}\left(\int_D \rho(z)^q d\nu(z)\right)^{\frac{1}{q}}\leq A
\end{equation}
for all discs $D$ in the plane. Here $d\nu(z)=dm(z)/\rho(z)^2$. We note that this is trivially satisfied if $p=2$. It is also satisfied for all $p$ if
$\rho(z)\simeq1$, as is the case in the classical Bargmann-Fock space. We now construct an example to show that there are situations where this condition
does not hold. As a first observation, since \eqref{Appart} is symmetric in $p$ and $q$, we can assume $p<2$. We note that if $D=D(\zeta,R)$ is of
sufficiently small radius, then $\rho(z)\simeq\rho(\zeta)$ for all $z\in D$ by \eqref{Lip}. Thus
\begin{multline*}
\frac{1}{|D|}\left(\int_D \rho(z)^p d\nu(z)\right)^\frac{1}{p}\left(\int_D \rho(z)^q d\nu(z)\right)^{\frac{1}{q}}\\
\simeq\frac{1}{|D|}\left(\rho(\zeta)^{p-2}\int_D dm(z)\right)^\frac{1}{p}\left(\rho(\zeta)^{q-2}\int_D dm(z)\right)^{\frac{1}{q}}=1
\end{multline*}
so that it suffices to check only discs of large radius.

We will take $\phi(z)=C_\gamma|z|^\gamma$ for some positive constants $\gamma$ and $C_\gamma$, which means that $\rho(z)\simeq\rho(0)$ for $z\in D(0)$ and
$\rho(z)\simeq|z|^{1-\frac{\gamma}{2}}$ for $z\not\in D(0)$. By choosing $C_\gamma$ appropriately, we may assume $\rho(0)=2$. We pick $R>\rho(0)$ and take
$D=D(0,R)$. Now
\begin{align*}
\left(\int_D \rho(z)^p d\nu(z)\right)^\frac{1}{p}&=\left(\int_{D(0)} \rho(z)^p d\nu(z)+\int_{D\backslash D(0)} \rho(z)^p d\nu(z)\right)^\frac{1}{p}\\
&\simeq\left(\rho(0)^{p-2}|D(0)|+\int_{\rho(0)}^R r^{(1-\frac{\gamma}{2})(p-2)}rdr\right)^\frac{1}{p}\\
&\simeq\left(\rho(0)^p+\frac{R^{p-\frac{\gamma p}{2}+\gamma}-\rho(0)^{p-\frac{\gamma p}{2}+\gamma}}{p-\frac{\gamma p}{2}+\gamma}\right)^\frac{1}{p}\\
&\simeq\left(R^{p-\frac{\gamma p}{2}+\gamma}\right)^\frac{1}{p}=R^{1-\frac{\gamma}{2}+\frac{\gamma}{p}}
\end{align*}
since $p-\frac{\gamma p}{2}+\gamma>0$ for $p<2$. We now chose some $\gamma$ such that $q-\frac{\gamma q}{2}+\gamma<0$. Then an identical computation gives
$$\left(\int_D \rho(z)^q d\nu(z)\right)^\frac{1}{q} \simeq\left(\rho(0)^q+\frac{R^{q-\frac{\gamma q}{2}+\gamma}-\rho(0)^{q-\frac{\gamma
q}{2}+\gamma}}{q-\frac{\gamma q}{2}+\gamma}\right)^\frac{1}{q} \simeq\rho(0)=2.$$ Therefore
$$\frac{1}{|D|}\left(\int_D \rho(z)^p d\nu(z)\right)^\frac{1}{q}\left(\int_D \rho(z)^q d\nu(z)\right)^{\frac{1}{q}}\simeq\frac{2}{R^2}R^{1-\frac{\gamma}{2}+\frac{\gamma}{p}}\simeq R^{-1-\frac{\gamma}{2}+\frac{\gamma}{p}}$$
which is only uniformly bounded if $-1-\frac{\gamma}{2}+\frac{\gamma}{p}<0$. However
$$-1-\frac{\gamma}{2}+\frac{\gamma}{p}=-1-\frac{\gamma}{2}+\gamma(1-\frac{1}{q})=-1+\frac{\gamma}{2}+\frac{\gamma}{q}$$
which we have assumed to be positive. This shows that there exist situations where $\rho^{p-2}$ is not an $A_p$ weight.

As before $\Lambda=\mathcal{Z}(g)$ will be the irregular lattice we are considering. Given a sequence $d\in\ell^{p}(\rho^{-1})$ we define, for each
$\lambda'\in\Lambda$,
\begin{equation}\label{defn}
B_{\lambda'}(d)=\sum_{\lambda\in\Lambda_{\lambda'}} \frac{d_{\lambda}}{(\lambda'-\lambda)^{2}}
\end{equation}
which we shall normally write as $B_{\lambda'}$, suppressing the dependence on $d$. It is clear that this is the discrete analogue of~\eqref{cnts}.
Lemma~\ref{usef} shows that, for $1\leq p\leq2$, this sum converges absolutely for each $\lambda'\in\Lambda$. Also, by Lemma~\ref{usef}, this sum
converges for $2<p<\infty$ if $t>1/2$ where $t$ is the constant occurring in Lemma~\ref{Christ}. Our main result is the following, which is proved using
the boundedness of~\eqref{cnts}.
\begin{thm}\label{discBA}
Fix $1<p<+\infty$ and suppose that $\rho^{p-2}$ is an $A_p$ weight. Define the operator
\begin{align*}
B:\ell^{p}(\rho^{-1})&\rightarrow\mathbb{C}^{\Lambda}\\
 d&\mapsto (B_{\lambda'})_{\lambda'\in\Lambda}
\end{align*}
where $B_{\lambda'}$ is given by~\eqref{defn}. Then $B$ is a bounded linear operator from $\ell^{p}(\rho^{-1})$ to $\ell^{p}(\rho)$ for $1<p\leq2$. If in
addition $t>1/2$ then the result also holds for $2<p<+\infty$. Here $t$ is the constant occurring in Lemma~\ref{Christ}.
\end{thm}
\begin{proof}
We first note that it is obvious that $B$ is linear, we are interested in showing that it indeed maps $\ell^{p}(\rho^{-1})$ to $\ell^{p}(\rho)$, and is a
bounded operator. Recall that the sets $D^{\delta_1}(\lambda)$ are pairwise disjoint. Suppose that $d\in\ell^{p}(\rho^{-1})$ and define
$f:\mathbb{C}\rightarrow\mathbb{C}$ by
$$f(z)=\frac{1}{\pi\delta_1^2}\sum_{\lambda\in\Lambda}\frac{d_{\lambda}}{\rho(\lambda)^2}\chi_{D^{\delta_1}(\lambda)}(z)$$
where $\chi_D$ is the characteristic function of the set $D$. Then clearly $f\in L^{p}(\rho^{p-2})$. In fact
$\|f\|_{L^{p}(\rho^{p-2})}^p\simeq\frac{\pi\delta_1^2}{\pi^p\delta_1^{2p}}\sum_{\lambda\in\Lambda}\left|\frac{d_\lambda}{\rho(\lambda)}\right|^p$ so that,
by our $A_p$ assumption, $T[f]\in L^{p}(\rho^{p-2})$ and indeed
$$\|T[f]\|_{L^{p}(\rho^{p-2})}\leq\|T\|\|f\|_{L^{p}(\rho^{p-2})}\simeq\|T\|\|d_\lambda\|_{\ell^{p}(\rho^{-1})}.$$
Now
\begin{align*}
T[f](\lambda')&=\lim_{\epsilon\rightarrow0}\int_{\mathbb{C}\backslash D(\lambda',\epsilon)}\frac{f(z)}{(\lambda'-z)^{2}}dm(z)\notag\\
&=\sum_{\lambda\in\Lambda_{\lambda'}}\int_{D^{\delta_1}(\lambda)}\frac{f(z)}{(\lambda'-z)^{2}}dm(z)+\lim_{\epsilon\rightarrow0}\int_{D^{\delta_1}(\lambda')\backslash D(\lambda',\epsilon)}\frac{f(z)}{(\lambda'-z)^{2}}dm(z)\notag\\
&=\sum_{\lambda\in\Lambda_{\lambda'}}\int_{D^{\delta_1}(\lambda)}\frac{d_{\lambda}}{\pi\delta_1^2\rho(\lambda)^2}\frac{1}{(\lambda'-z)^{2}}dm(z)\notag\\
&=\sum_{\lambda\in\Lambda_{\lambda'}} \frac{d_{\lambda}}{(\lambda'-\lambda)^{2}}\notag\\
&=B_{\lambda'},
\end{align*}
since $f$ is constant on $D^{\delta_1}(\lambda)$ for each $\lambda\in\Lambda$ and the average value of a harmonic function on a disk is the value at the
centre. Fix $0<\delta<\delta_1$. It is obvious that
\begin{align}\label{BAmain}
|\rho(\lambda')B_{\lambda'}|^p=\rho(\lambda')^p|T[f](\lambda')|^p\lesssim&\rho(\lambda')^p\Bigl|T[f](\lambda')-\frac{1}{\pi\delta^2\rho(\lambda')^2}\int_{D^{\delta}(\lambda')}T[f](\zeta)dm(\zeta)\Bigr|^p\notag\\
&+\rho(\lambda')^p\Bigl|\frac{1}{\pi\delta^2\rho(\lambda')^2}\int_{D^{\delta}(\lambda')}T[f](\zeta)dm(\zeta)\Bigr|^p
\end{align}
and we shall estimate these terms separately. The second term is especially easy to bound since, by Jensen's inequality,
\begin{align}
\rho(\lambda')^p\left|\frac{1}{\pi\delta^2\rho(\lambda')^2}\int_{D^{\delta}(\lambda')}T[f](\zeta)dm(\zeta)\right|^p&\leq\rho(\lambda')^p\frac{1}{\pi\delta^2\rho(\lambda')^2}\int_{D^{\delta}(\lambda')}|T[f](\zeta)|^p dm(\zeta)\notag\\
&\simeq\frac{1}{\pi\delta^2}\int_{D^{\delta}(\lambda')}|T[f](\zeta)|^p\frac{dm(\zeta)}{\rho(\zeta)^{2-p}}\label{BAeasy}.
\end{align}

We now estimate the first term. Applying the definitions and computing gives
\begin{multline*}
T[f](\lambda')-\frac{1}{\pi\delta^2\rho(\lambda')^2}\int_{D^{\delta}(\lambda')}T[f](\zeta)dm(\zeta)=\frac{1}{\pi\delta^2\rho(\lambda')^2}\int_{D^{\delta}(\lambda')}T[f](\lambda')-T[f](\zeta)dm(\zeta)\\
=\frac{1}{\pi\delta^2\rho(\lambda')^2}\int_{D^{\delta}(\lambda')}\Bigl( \lim_{\epsilon\rightarrow0}\int_{\mathbb{C}\backslash D(\lambda',\epsilon)}\frac{f(z)}{(\lambda'-z)^{2}}dm(z) -\lim_{\epsilon\rightarrow0}\int_{\mathbb{C}\backslash D(\zeta,\epsilon)}\frac{f(z)}{(\zeta-z)^{2}}dm(z)\Bigr)dm(\zeta)\\
=\frac{1}{\pi\delta^2\rho(\lambda')^2}\int_{D^{\delta}(\lambda')}\Bigl\{\int_{\mathbb{C}\backslash D^{\delta_1}(\lambda')}f(z)\Bigl[\frac{1}{(\lambda'-z)^{2}}-\frac{1}{(\zeta-z)^{2}}\Bigr]dm(z)\\
+\lim_{\epsilon\rightarrow0}\int_{D^{\delta_1}(\lambda')\backslash
D(\lambda',\epsilon)}\frac{f(z)}{(\lambda'-z)^{2}}dm(z)-\lim_{\epsilon\rightarrow0}\int_{D^{\delta_1}(\lambda')\backslash
D(\zeta,\epsilon)}\frac{f(z)}{(\zeta-z)^{2}}dm(z)\Bigr\}dm(\zeta).
\end{multline*}
We shall bound each of these three terms separately. First note that, by symmetry,
$$\lim_{\epsilon\rightarrow0}\int_{D^{\delta_1}(\lambda')\backslash D(\lambda',\epsilon)}\frac{f(z)}{(\lambda'-z)^{2}}dm(z)=0.$$
Note also that if
$$\int_{D^{\delta}(\lambda')}\int_{\mathbb{C}\backslash
D^{\delta_1}(\lambda')}f(z)\Bigl[\frac{1}{(\lambda'-z)^{2}}-\frac{1}{(\zeta-z)^{2}}\Bigr]dm(z)dm(\zeta)$$
is absolutely convergent then it vanishes similarly, since we may
apply Fubini's theorem. But
\begin{multline*}
\Bigl|\int_{D^{\delta}(\lambda')}\int_{\mathbb{C}\backslash
D^{\delta_1}(\lambda')}f(z)\Bigl[\frac{1}{(\lambda'-z)^{2}}-\frac{1}{(\zeta-z)^{2}}\Bigr]dm(z)dm(\zeta)\Bigr|\\=
\Bigl|\int_{D^{\delta}(\lambda')}\int_{\mathbb{C}\backslash
D^{\delta_1}(\lambda')}f(z)\Bigl[\frac{(\zeta+\lambda'-2z)(\zeta-\lambda')}{(\lambda'-z)^{2}(\zeta-z)^{2}}\Bigr]dm(z)dm(\zeta)\Bigr|\\
\lesssim\int_{D^{\delta}(\lambda')}|\zeta-\lambda'|dm(\zeta)\int_{\mathbb{C}\backslash
D^{\delta_1}(\lambda')}\frac{|f(z)|}{|\lambda'-z|^3}dm(z)
\end{multline*}
since for $\zeta\in D^{\delta}(\lambda')$ and $z\in \mathbb{C}\backslash D^{\delta_1}(\lambda')$ we have $|\zeta-z|\simeq|\lambda'-z|$. The integral in
$\zeta$ is clearly finite. It remains only to estimate
$$\int_{\mathbb{C}\backslash
D^{\delta_1}(\lambda')}\frac{|f(z)|}{|\lambda'-z|^3}dm(z)=\sum_{\lambda\in\Lambda_{\lambda'}}\frac{|d_{\lambda}|}{\pi\delta_1^2\rho(\lambda)^2}\int_{D^{\delta_1}(\lambda)}\frac{dm(z)}{|\lambda'-z|^3}
\simeq\sum_{\lambda\in\Lambda_{\lambda'}}\frac{|d_{\lambda}|}{|\lambda'-\lambda|^3}$$
which we have already seen is finite under our hypothesis, in
Lemma~\ref{usef}. We consequently have
\begin{multline*}
T[f](\lambda')-\frac{1}{\pi\delta^2\rho(\lambda')^2}\int_{D^{\delta}(\lambda')}T[f](\zeta)dm(\zeta)\\
=\frac{1}{\pi\delta^2\rho(\lambda')^2}\int_{D^{\delta}(\lambda')}-\lim_{\epsilon\rightarrow0}\int_{D^{\delta_1}(\lambda')\backslash D(\zeta,\epsilon)}\frac{f(z)}{(\zeta-z)^{2}}dm(z)dm(\zeta)\\
=-\frac{d_{\lambda'}}{\pi^2\delta^2\delta_1^2\rho(\lambda')^4}\int_{D^{\delta}(\lambda')}\lim_{\epsilon\rightarrow0}\int_{D^{\delta_1}(\lambda')\backslash
D(\zeta,\epsilon)}\frac{dm(z)}{(\zeta-z)^{2}}dm(\zeta)
\end{multline*}
Now the inner integral does not change in value for $\epsilon\leq(\delta_1-\delta)\rho(\lambda')$. We therefore have
\begin{multline}\label{BAcentre}
\bigl|T[f](\lambda')-\frac{1}{\pi\delta^2\rho(\lambda')^2}\int_{D^{\delta}(\lambda')}T[f](\zeta)dm(\zeta)\bigr|\\
\leq\frac{|d_{\lambda'}|}{\pi^2\delta^2\delta_1^2\rho(\lambda')^4}\int_{D^{\delta}(\lambda')}\int_{D^{\delta_1}(\lambda')\backslash D(\zeta,(\delta_1-\delta)\rho(\lambda'))}\frac{dm(z)}{|\zeta-z|^{2}}dm(\zeta)\\
\leq\frac{|d_{\lambda'}|}{\pi^2\delta^2\delta_1^2\rho(\lambda')^4}\frac{1}{(\delta_1-\delta)^2\rho(\lambda')^2}|D^{\delta}(\lambda')||D^{\delta_1}(\lambda')\backslash
D(\zeta,(\delta_1-\delta)\rho(\lambda'))|\simeq\frac{|d_{\lambda'}|}{\rho(\lambda')^2}.
\end{multline}
Inserting~\eqref{BAeasy} and~\eqref{BAcentre} into ~\eqref{BAmain} gives finally that
\begin{align*}
\sum_{\lambda'\in\Lambda}|\rho(\lambda')B_{\lambda'}|^p&\lesssim\sum_{\lambda'\in\Lambda}\left(\int_{D^{\delta}(\lambda')}|T[f](\zeta)|^p\frac{dm(\zeta)}{\rho(\zeta)^{2-p}}+\frac{|d_{\lambda'}|^p}{\rho(\lambda')^p}\right)\\
&\leq\int_\C|T[f](\zeta)|^p\frac{dm(\zeta)}{\rho(\zeta)^{2-p}}+\sum_{\lambda'\in\Lambda}\frac{|d_{\lambda'}|^p}{\rho(\lambda')^p}\\
&=\|T[f]\|_{L^p(\rho^{p-2})}^p+\|d\|_{\ell^{p}(\rho^{-1})}^p\lesssim(1+\|T\|^p)\|d\|_{\ell^{p}(\rho^{-1})}^p
\end{align*}
so that $B$ is indeed a bounded operator as claimed.
\end{proof}

\section{Proofs}
We shall use the same notation as before. We write $d_\lambda=c_\lambda/g'(\lambda)$ which, by virtue of the growth conditions on $g$, satisfies
$\left(d_\lambda/\rho(\lambda)\right)_{\lambda\in\Lambda}\in\ell^p$. We are essentially going to give a unified proof of Theorems~\ref{main2},
\ref{main1}, \ref{mainbetw1n2} and \ref{mainbig2}. We shall refer to an integer $N$ which should be thought of as $2$ for the case of
Theorems~\ref{main2}, \ref{main1}, \ref{mainbetw1n2} and \ref{mainbig2}(i) and (ii), but to be the integer $N$ appearing in the statement of
Theorem~\ref{mainbig2}(iii). We also note that if $N=2$ and $\rho^{p-2}$ satisfies the $A_p$ condition then we may apply Theorem~\ref{discBA}. We begin by
showing the necessity of the stated results.
\begin{proof}[Proof of the necessity]
We have already remarked in \eqref{basicnec} that (a) follows from the Plancherel-Polya type estimate. We define $\gamma(R)$ as in the proof of
Lemma~\ref{lemrepinf}. Computing, for any $\lambda'\in\Lambda$,
$$\frac{1}{2\pi i}\int_{\gamma(R)}\frac{f(w)}{g(w)(w-\lambda')^n}dw$$ in exactly the same manner as in the proof of Lemma~\ref{lemrepinf}, where $1\leq n\leq N$, shows that
$$\text{\textup{p.v.}}\sum_{\lambda\in\Lambda_{\lambda'}}\frac{d_\lambda}{(\lambda-\lambda')^n}$$
is well-defined. Fix some $0<\delta<\delta_1$ and some integer $0\leq k<N$. Define $\omega_k=e^{2\pi ik/N}$ and $z_{\lambda'}^k=\lambda'+\delta
\omega_k\rho(\lambda')$. Then, for each $k$, $(z_{\lambda'}^k)_{\lambda'\in\Lambda}$ is a $\rho$-separated sequence that is bounded away from $\Lambda$ in
the distance $d_\phi$. \eqref{rhosepseries} implies that
$$\sum_{\lambda'\in\Lambda}\left|\frac{f(z_{\lambda'}^k)}{g(z_{\lambda'}^k)}\right|^p<+\infty.$$
Replacing $z$ by $\delta \omega_k\rho(\lambda')$ and $\lambda$ by
$\lambda-\lambda'$ in Identity~\eqref{tayl} yields
$$\frac{1}{z_{\lambda'}^k-\lambda}+\frac{1}{\lambda-\lambda'}+\frac{\delta \omega_k\rho(\lambda')}{(\lambda-\lambda')^2}+\cdots+\frac{(\delta \omega_k\rho(\lambda'))^{n-1}}{(\lambda-\lambda')^n}=\frac{(\delta \omega_k\rho(\lambda'))^n}{(\lambda-\lambda')^n(z_{\lambda'}^k-\lambda)}.$$
Consequently, invoking~\eqref{repfinite}, we compute that
\begin{multline*}
\frac{f(z_{\lambda'}^k)}{g(z_{\lambda'}^k)}+\text{\textup{p.v.}}\sum_{\lambda\in\Lambda_{\lambda'}}d_\lambda\left(\frac{1}{\lambda-\lambda'} +\frac{\delta
\omega_k\rho(\lambda')}{(\lambda-\lambda')^2}+\cdots+\frac{(\delta \omega_k\rho(\lambda'))^{N-1}}{(\lambda-\lambda')^N}\right)
\\=\frac{d_{\lambda'}}{\delta \omega_k\rho(\lambda')}+\text{\textup{p.v.}}\sum_{\lambda\in\Lambda_{\lambda'}}\frac{d_\lambda(\delta
\omega_k\rho(\lambda'))^N}{(\lambda-\lambda')^N(z_{\lambda'}^k-\lambda)}.
\end{multline*}
Hence
\begin{multline*}
\sum_{\lambda'\in\Lambda}\Bigl|\text{\textup{p.v.}}\sum_{\lambda\in\Lambda_{\lambda'}}d_\lambda\bigl(\frac{1}{\lambda-\lambda'}
+\frac{\delta
\omega_k\rho(\lambda')}{(\lambda-\lambda')^2}+\cdots+\frac{(\delta
\omega_k\rho(\lambda'))^{N-1}}{(\lambda-\lambda')^N}\bigr)\Bigr|^p
\\ \lesssim\sum_{\lambda'\in\Lambda}\Bigl\{\bigl(\delta^{N}\rho(\lambda')^{N}\text{\textup{p.v.}}\sum_{\lambda\in\Lambda_{\lambda'}}\frac{|d_\lambda|}{|\lambda-\lambda'|^N|z_{\lambda'}^k-\lambda|}\bigr)^p+\bigl|\frac{d_{\lambda'}}{\delta\rho(\lambda')}\bigr|^p+\bigl|\frac{f(z_{\lambda'}^k)}{g(z_{\lambda'}^k)}\bigr|^p\Bigr\}.
\end{multline*}
We know that the second and third terms are summable, it remains only to estimate the first. All of the terms in this sum are positive, so we may ignore
the principal value. Moreover $|z_{\lambda'}^k-\lambda|\simeq|\lambda-\lambda'|$, so that Lemma~\ref{usef} shows that this sum is convergent. (It is here
that the value of $N$ is important.) Taking now linear combinations over different $k$ completes the proof, for example
\begin{multline*}
\sum_{\lambda'\in\Lambda}\bigl|\text{\textup{p.v.}}\sum_{\lambda\in\Lambda_{\lambda'}}\frac{d_\lambda}{\lambda-\lambda'}\bigr|^p=\\
\sum_{\lambda'\in\Lambda}\Bigl|\text{\textup{p.v.}}\sum_{\lambda\in\Lambda_{\lambda'}}d_\lambda\bigl(\frac{1}{N}\sum_{k=0}^{N-1}\frac{1}{\lambda-\lambda'}
+\frac{\delta
\omega_k\rho(\lambda')}{(\lambda-\lambda')^2}+\cdots+\frac{(\delta
\omega_k\rho(\lambda'))^{N-1}}{(\lambda-\lambda')^N}\bigr)\Bigr|^p<+\infty
\end{multline*}
\end{proof}
We now turn to the proof of the sufficiency, which is similar. We use the same notation.
\begin{proof}[Proof of the sufficiency]
We wish to construct a function, that solves the interpolation problem $f|\Lambda=c$. As in the proof to Lemma~\ref{lemrepinf}, the na\"{\i}ve attempt at
Lagrange interpolation is not in general convergent. We modify in the exact same manner, and a similar argument shows that
$$G(z)=\frac{d_0}{z}+\sum_{\lambda\in\Lambda_0}d_\lambda\left(\frac{1}{z-\lambda}+\frac{1}{\lambda}+\frac{z}{\lambda^2}+\cdots+\frac{z^{N-1}}{\lambda^N}\right)$$
defines a meromorphic function on $\C$. (Here we invoke Lemma~\ref{usef} to see the series is convergent, which once more determines the value of $N$.)
Hence
$$G(z)+\sum_{k=1}^N z^{k-1}\text{\textup{p.v.}} \sum_{\lambda\in\Lambda_0}\frac{d_\lambda}{\lambda^k}=
\text{\textup{p.v.}}\sum_{\lambda\in\Lambda}\frac{d_\lambda}{z-\lambda}$$ is a well-defined meromorphic function. It follows that
$f(z)=g(z)\text{\textup{p.v.}}\sum_{\lambda\in\Lambda}\frac{d_\lambda}{z-\lambda}$ is an entire function satisfying $f(\lambda)=c_\lambda$. It remains to
show that $f\in\Fpa$. We must show the following integral is finite (Recall that $Q_\lambda=\{z\in\mathbb{C}:d_\phi(z,\Lambda)=d_\phi(z,\lambda)\}$):
\begin{align*}
\int_\mathbb{C}|f(z)|^p e^{-p\phi(z)}\frac{dm(z)}{\rho(z)^2}&=\sum_{\lambda'\in\Lambda}\int_{Q_{\lambda'}}|f(z)|^p e^{-p\phi(z)}\frac{dm(z)}{\rho(z)^2}\\
&=\sum_{\lambda'\in\Lambda}\int_{Q_{\lambda'}}\bigl|g(z) e^{-\phi(z)}\text{\textup{p.v.}}\sum_{\lambda\in\Lambda}\frac{d_\lambda}{z-\lambda}\bigr|^p\frac{dm(z)}{\rho(z)^2}\\
&\simeq\sum_{\lambda'\in\Lambda}\int_{Q_{\lambda'}}\Bigl|d_\phi(z,\lambda')\bigl(\frac{d_{\lambda'}}{z-\lambda'}+\text{\textup{p.v.}}\sum_{\lambda\in\Lambda_{\lambda'}}\frac{d_\lambda}{z-\lambda}\bigr)\Bigr|^p\frac{dm(z)}{\rho(z)^2}\\
&\lesssim\sum_{\lambda'\in\Lambda}\bigl|\frac{d_{\lambda'}}{\rho(\lambda')}\bigr|^p+\sum_{\lambda'\in\Lambda}\int_{Q_{\lambda'}}\bigl|\text{\textup{p.v.}}\sum_{\lambda\in\Lambda_{\lambda'}}\frac{d_\lambda}{z-\lambda}\bigr|^p\frac{dm(z)}{\rho(z)^2},
\end{align*}
where we have used the fact that $d_\phi(z,\lambda')\simeq|z-\lambda'|/\rho(\lambda')\lesssim1$ for $z\in Q_{\lambda'}$ by Lemma~\ref{lem4}. The first
term is finite by hypothesis, so we need only bound the second. Once more we use Identity~\eqref{tayl} which suitably modified yields
$$\frac{1}{z-\lambda}=\frac{(z-\lambda')^N}{(z-\lambda)(\lambda-\lambda')^N}-\frac{1}{\lambda-\lambda'}-\frac{z-\lambda'}{(\lambda-\lambda')^2}-\cdots-\frac{(z-\lambda')^{N-1}}{(\lambda-\lambda')^N}$$
whence
\begin{align*}
\sum_{\lambda'\in\Lambda}\int_{Q_{\lambda'}}\bigl|\text{\textup{p.v.}}\sum_{\lambda\in\Lambda_{\lambda'}}\frac{d_\lambda}{z-\lambda}\bigr|^p\frac{dm(z)}{\rho(z)^2}\lesssim&\sum_{\lambda'\in\Lambda}\int_{Q_{\lambda'}}\bigl|\text{\textup{p.v.}}\sum_{\lambda\in\Lambda_{\lambda'}}\frac{d_\lambda(z-\lambda')^N}{(z-\lambda)(\lambda-\lambda')^N}\bigr|^p\\
&+\sum_{n=1}^N\bigl|\text{\textup{p.v.}}\sum_{\lambda\in\Lambda_{\lambda'}}\frac{d_\lambda(z-\lambda')^{n-1}}{(\lambda-\lambda')^n}\bigr|^p\frac{dm(z)}{\rho(z)^2}.
\end{align*}
Now for $z\in Q_{\lambda'}$ we have $|z-\lambda'|\lesssim\rho(\lambda')$ and $|z-\lambda|\simeq|\lambda'-\lambda|$. Hence, by ~\eqref{intQ},
\begin{align*}
\sum_{\lambda'\in\Lambda}\int_{Q_{\lambda'}}\bigl|\text{\textup{p.v.}}\sum_{\lambda\in\Lambda_{\lambda'}}\frac{d_\lambda}{z-\lambda}\bigr|^p\frac{dm(z)}{\rho(z)^2}\lesssim&\sum_{\lambda'\in\Lambda}\bigl(\rho(\lambda')^N\text{\textup{p.v.}}\sum_{\lambda\in\Lambda_{\lambda'}}\frac{|d_\lambda|}{|\lambda-\lambda'|^{N+1}}\bigr)^p+\\
&\sum_{n=1}^N\sum_{\lambda'\in\Lambda}\bigl|\rho(\lambda')^{n-1}\text{\textup{p.v.}}\sum_{\lambda\in\Lambda_{\lambda'}}\frac{d_\lambda}{(\lambda-\lambda')^n}\bigr|^p.
\end{align*}
The first term is finite by Lemma~\ref{usef} (again the value of $N$ is important here), the remainder by hypothesis. This completes the proof.
\end{proof}
The proof of Theorem~\ref{maininf} is similar and omitted.

\end{document}